\documentclass[a4paper,12pt]{amsart}

\makeindex

\usepackage{xcolor}
\usepackage[
colorlinks=true,
urlcolor=purple,
linkcolor=purple!87!black,
pdfborder={0 0 0}
]{hyperref}

\usepackage{amsfonts,graphics,amsmath,amsthm,amsfonts,amscd,amssymb,amsmath,latexsym,euscript, enumerate,kotex}
\usepackage{epsfig,url}
\usepackage{flafter}
\usepackage[all,cmtip,line]{xy}
\usepackage{array}
\usepackage[english]{babel}
\usepackage{overpic}
\usepackage{subfig}
\usepackage{multirow}
\usepackage{tikz-cd}
\usepackage{theoremref}

\usepackage{wrapfig}
\usepackage[shortlabels]{enumitem}
\setlist[enumerate, 1]{1\textsuperscript{o}}

\newtheorem{thm}{Theorem}[section]
\newtheorem{lem}[thm]{Lemma}
\newtheorem{prop}[thm]{Proposition}

\newtheorem{coro}[thm]{Corollary}

\theoremstyle{definition} 
\newtheorem{defi}[thm]{Definition}
\newtheorem{definition-lemma}[thm]{Definition-Lemma}

\theoremstyle{remark}
\newtheorem{rmk}[thm]{Remark}

\newtheorem*{ack}{Acknowledgments}

\numberwithin{equation}{section}

\newcommand{\C}{\mathbb{C}}
\newcommand{\R}{\mathbb{R}}
\newcommand{\Z}{\mathbb{Z}}

\newcommand{\Q}{\mathbb{Q}}

\newcommand{\m}{\mathfrak{m}}
\def\P{\mathbb{P}}

\DeclareRobustCommand{\O}{\mathcal{O}}

\DeclareMathOperator{\ord}{ord}

\DeclareMathOperator{\lct}{lct}

\newcommand{\B}{\mathbf B}

\def\Bs{\operatorname{Bs}}

\def\Supp{\operatorname{Supp}}

\newcommand{\floor}[1]{\left\lfloor #1 \right\rfloor}

\let\oldframe\frame
\renewcommand\frame[1][allowframebreaks]{\oldframe[#1]}

\title[On diminished multiplier ideal and the termination of flips]
{On diminished multiplier ideal and the termination of flips}
\begin{document}

\author{Donghyeon Kim}
\address{Department of mathematics, Yonsei University, 50 Yonsei-Ro, Seodaemun-Gu, Seoul 03722, Korea}
\email{whatisthat@yonsei.ac.kr}

\date{\today}

\begin{abstract}
In this paper, we develop a theory of diminished multiplier ideals on singular varieties which was introduced by Hacon, and developed by Lehmann. We prove a result regarding the termination of certain type of flips with scaling of an ample divisor if the Cartier index is bounded, and if $\kappa_{\sigma}(K_X+\Delta)\ge \dim X-1$ holds. The proof uses a theory of diminished multiplier ideal.
\end{abstract}

\maketitle

\section{Introduction}
The goal of this paper is twofold: first, to extend the notion of the diminished multiplier ideal to the singular case, and second, to prove a result on the termination of flips. The former development underpins our proof of the latter.

Hacon introduced the notion of diminished multiplier ideal $\mathcal{J}_{\sigma}(\|D\|)$ and later, Lehmann developed a theory of $\mathcal{J}_{\sigma}(\|D\|)$ on smooth varieties (cf. \cite{Hac04}, \cite{Leh14}). Asymptotic multiplier ideals $\mathcal{J}(\|D\|)$ are defined for big or effective divisors, and can be considered as a ``continuous" extension of $\mathcal{J}(\|D\|)$ for big divisors to pseudo-effective divisor. The first goal of this paper is to further develop a theory of diminished multiplier ideal $\mathcal{J}_{\sigma}(X,\Delta;\|D\|)$ associated to any klt pair $(X,\Delta)$ and any pseudo-effective $\Q$-divisor $D$ on $X$. For the definition of asymptotic multiplier ideals $\mathcal{J}(X,\Delta;\|D\|)$, see Section \ref{Section 4}, and for the definition of diminished multiplier ideals $\mathcal{J}_{\sigma}(X,\Delta;\|D\|)$, see \thref{realdef}. The three main theorems in this direction are presented below:

\begin{thm}\thlabel{imain1}
Let $(X,\Delta)$ be a projective klt pair, and $D$ a big $\Q$-divisor. Then we have
$$ \mathcal{J}_{\sigma}(X,\Delta;\|D\|)=\mathcal{J}(X,\Delta;\|D\|).$$
\end{thm}

\thref{imain1} is a generalization of Corollary 6.12 in \cite{Leh14}.

\begin{thm}\thlabel{imain2}
Let $(X,\Delta)$ be a projective klt pair, and $D$ a $\Q$-Cartier $\Q$-divisor on $X$. Then we have that
$$ \B_-(D)=\bigcup_m Z(\mathcal{J}_{\sigma}(X,\Delta;\|mD\|))_{\mathrm{red}}.$$
For the definition of $\B_-$, see \thref{base locus}.
\end{thm}

\thref{imain2} is a generalization of Lemma 6.4 in \cite{Leh14}.

The following can be considered as a generalization of Nadel vanishing theorem.

\begin{thm}\thlabel{imain3}
Let $(X,\Delta)$ be a projective klt pair such that $\Delta$ is $\Q$-Cartier, and suppose that $D$ is a pseudo-effective $\Q$-divisor on $X$. Then we have
$$ H^i(X,\O_X(L)\otimes \mathcal{J}_{\sigma}(X,\Delta;\|D\|))=0\text{ for }i>\dim X-\kappa_{\sigma}(D),$$
where $L$ is a Cartier divisor on $X$ such that $L\sim_{\Q}K_X+\Delta+D$ holds. For the definition of $\kappa_{\sigma}$, see \thref{Kodaira}.
\end{thm}

We now apply the preceding results to the problem of flip termination, a longstanding topic in birational geometry. A common strategy in studying termination is to introduce an appropriate invariant that remains well-behaved along the Minimal Model Program (MMP). For instance, Shokurov introduced the \emph{difficulty function} in \cite{Sho85} and used it to prove the termination of $3$-fold terminal flips. His idea gave rise to the notion of \emph{special termination} (cf. \cite{Fuj07}), and in \cite{Sho04}, he further showed that the \emph{ACC condition} and the lower semi-continuity of \emph{minimal log discrepancy} (mld) together imply the termination of flips.

In our setting, under the assumptions that $\kappa_{\sigma}(K_X+\Delta)\ge \dim X-1$ and that the index of $K_X+\Delta$ is bounded along the MMP, we propose another invariant, namely
$$ h^1(X,\O_X(m(K_X+\Delta)))$$
for some positive integer $m$. Building on Theorems~\ref{imain2} and \ref{imain3}, we use this cohomological invariant to prove the following result on the termination of flips.



\begin{thm}\thlabel{ithm}
Let $(X,\Delta)$ be a $\Q$-factorial projective klt pair, and suppose that
$$ (X_0,\Delta_0):=(X,\Delta)\overset{f_1}{\dashrightarrow} (X_1,\Delta_1)\overset{f_2}{\dashrightarrow} \cdots$$
is a $(K_X+\Delta)$-MMP with scaling of an ample divisor. If
\begin{itemize}
    \item[(a)] all of $f_i$ are flips of type $(1,\dim X-2)$,
    \item[(b)] there is a positive integer $m$ such that $m(K_{X_i}+\Delta_i)$ are Cartier divisors for all $i$, and
    \item[(c)] $\kappa_{\sigma}(K_X+\Delta)\ge \dim X-1$ holds,
\end{itemize}
then the MMP terminates.
\end{thm}

Note that the $\Q$-factoriality of $X$ is required only to use \thref{imain3}. Moreover, \cite[Lemma 4.4.1]{Sho96} induces a new proof of the termination of $\Q$-factorial $3$-fold klt flips with scaling of an ample divisor if $\kappa_{\sigma}(K_X+\Delta)\ge 2$. Furthermore, if $\kappa_{\sigma}(K_X+\Delta)=\dim X$ (i.e., $K_X+\Delta$ is big), then the termination of flips with scaling of an ample divisor automatically follows from \cite[Corollary 1.4.2]{BCH+10}.

The result may be useful for investigating another related problems on termination of flips. For instance, it is known by \cite{LX23} that a certain termination of flips is enough to guarantee the termination of a larger class of flips. Moreover, we also expect that \thref{ithm} provide some hint of the proof of the termination of flips (cf. \cite[Theorem 5-1-15]{KMM87}).

A sketch of the proof is given below. Firstly, for a fixed positive integer $m$ such that all $m(K_{X_i}+\Delta_i)$ are Cartier (such an $m$ exists by $(b)$), we show that
$$d_i:=h^1(X_i,\O_{X_i}(2m(K_{X_i}+\Delta_i)))$$
is non-increasing. Then since $d_i\ge 0$ for any $i$, we can find $i$ such that $d_{i-1}=d_i$ holds. Moreover, if $d_{i-1}=d_i$ holds, then by $(a)$ and $(c)$, we obtain a vanishing of higher direct image of the flipping contraction. We can translate the vanishing into the vanishing of cohomology of fibre of flipping contraction, and it leads to a contradiction if the MMP does not terminate. Note that to use the condition $(c)$, we develop a theory of diminished multiplier ideals for singular varieties.

The remaining of the paper is organized as follows. We begin Section 2 by introducing notions and proving lemmas which are used in the paper. In Section 3, we prove a version of \cite[Theorem 7.3]{JM12}. Section 4 is devoted to developing a theory of diminished multiplier ideal for singular varieties, and proving three main theorems on the notion. In Section 5, we prove \thref{ithm}.

\begin{ack}
The author thanks his advisor Sung Rak Choi for his comments, questions, and discussions. He is grateful for his encouragement and support. The author is grateful to Christopher Hacon for his helpful comments on earlier versions of this paper. The authors are grateful to referees for careful reading and suggestions on improving the paper.
\end{ack}

\section{Preliminaries}
The aim of this section is to introduce the notions and lemmas which will be used in later sections. Let us collect the basic notions.

\begin{itemize}
    \item All varieties are integral schemes of finite type over the complex number field $\C$.
    \item For a normal variety $X$, and a boundary $\Q$-Weil divisor $\Delta$ on $X$, we say $(X,\Delta)$ is a \emph{pair} if $K_X+\Delta$ is a $\Q$-Cartier divisor.
    \item Let $(X,\Delta)$ be a pair, and $E$ a prime divisor over $X$. Suppose $f:X'\to X$ is a log resolution of $(X,\Delta)$ in which $E$ is an $f$-exceptional divisor, and $\Delta_{X'}$ is a divisor on $X'$ such that
    $$ K_{X'}+\Delta_{X'}=f^*(K_X+\Delta)$$
    holds. The \emph{log discrepancy} $A_{X,\Delta}(E)$ is defined by
    $$ A_{X,\Delta}(E):=1+\mathrm{mult}_E(-\Delta_{X'}).$$
    Moreover, $(X,\Delta)$ is \emph{klt} if for any prime divisor $E$ over $X$, $A_{X,\Delta}(E)>0$ holds.
    \item For a noetherian scheme $X$, and a closed subscheme $Z\subseteq X$, let us denote by $\mathcal{I}_Z$ the corresponding ideal sheaf of $Z$. Moreover, for an ideal sheaf $\mathcal{I}$ of $X$, $Z(\mathcal{I})$ denotes the corresponding closed subscheme of $X$.
    \item For a normal projective variety $X$, and a Cartier divisor $D$ on $X$, $\mathrm{Bs}(|D|)$ denotes the \emph{base locus} of $|D|$.
    \item Let $X$ be a variety, and $S\subseteq X$ a subset of $X$. A closed subvariety $Z$ of $X$ is an \emph{irreducible component} of $S$ if there is no other closed subvariety $Z'$ of $X$ such that $Z'\subseteq S$, and $Z\subsetneq Z'$.
    \item Let $X$ be a variety. Then $K(X)$ denotes its function field. Moreover, let $\eta\in X$ be a (not necessarily closed) point. Then $K(\eta)$ denotes the function field of $\overline{\{\eta\}}$
\end{itemize}

Let us recall the definition of \emph{numerical Kodaira dimension}.

\begin{defi}[{cf. \cite[2.5 Definition in Chapter V]{Nak04}}] \thlabel{Kodaira}
Let $X$ be a normal projective variety, and $D$ a pseudo-effective $\Q$-divisor on $X$. For an ample divisor $A$ on $X$, we define
$$ \kappa_{\sigma}(D):=\min\left\{i\in \Z\Bigg|\limsup_{n\to \infty}\frac{\dim H^0(X,\O_X(mD+A))}{m^i}<\infty\right\}.$$
The definition does not depend on the choice of $A$.
\end{defi}

For related results on the numerical Kodaira dimension, see \cite[Chapter V]{Nak04}.

\medskip

We recall the definition of negative contractions and flips.

\begin{defi}
Let $X$ be a normal projective variety, and $f:X\dashrightarrow X_0$ a birational map. Let $D$ be a $\Q$-Cartier divisor on $X$. We say that $f$ is a \emph{$D$-negative contraction} if there is a common resolution
$$ 
\begin{tikzcd}
& X' \ar["g"']{ld}\ar["h"]{rd}& \\
X \ar["f"',dashed]{rr}& & X_0
\end{tikzcd}
$$
such that $g^*D=h^*f_*D+E$ for some effective $h$-exceptional divisor $E$.
\end{defi}

\begin{defi}
Let $X$ be a normal projective variety, and $f:X\dashrightarrow X^+$ a birational map. Let $D$ be a $\Q$-Cartier divisor on $X$. We say that $f$ is a \emph{$D$-flip} if there is a diagram
$$
\begin{tikzcd}
X\ar["f",dashed]{rr}\ar["\varphi"']{rd}& &X^+ \ar["\varphi^+"]{ld}\\
& Z&
\end{tikzcd}
$$
such that the Picard number of $\varphi,\varphi^+$ satisfy $\rho(X/Z)=\rho(X^+/Z)=1$, $D$ is anti $\varphi$-ample, and $f_*D$ is $\varphi^+$-ample. We call the diagram above the \emph{flipping diagram} of $f$.
\end{defi}

Note that any $D$-flip is also a $D$-negative contraction by the Negativity lemma (cf. \cite[Lemma 3.39]{KM98}). Let us prove the following lemma:

\begin{lem}\thlabel{forapp}
Let $X$ be a normal projective variety, $D$ a pseudo-effective $\Q$-divisor, and $f:X\dashrightarrow X_0$ a $D$-negative birational contraction. Then $\kappa_{\sigma}(D)=\kappa_{\sigma}(f_*D)$.
\end{lem}

\begin{proof}
Let
$$ 
\begin{tikzcd}
& X' \ar["g"']{ld}\ar["h"]{rd}& \\
X \ar["f"',dashed]{rr}& & X_0
\end{tikzcd}
$$
be a resolution of indeterminacies such that $g^*D=h^*f_*D+E$ for some effective $h$-exceptional divisor $E$. Then by remarks in \cite[2.2]{LP20}, we obtain that
$$
\kappa_{\sigma}(D)=\kappa_{\sigma}(g^*D)=\kappa_{\sigma}(h^*f_*D+E)=\kappa_{\sigma}(f_*D).
$$
\end{proof}

Let us define the \emph{type} of a flip.

\begin{defi}
Let $X$ be a normal variety, $D$ a Cartier divisor on $X$, $f:X\dashrightarrow X^+$ a $D$-flip, and
$$ 
\begin{tikzcd}
X\ar[dashed]{rr}\ar["\varphi"']{rd}& &X^+ \ar["\varphi^+"]{ld}\\
& Z&
\end{tikzcd}
$$
the flipping diagram of $f$. We say that $f$ is of type $(d,d')$ if $d=\dim \mathrm{Exc}(\varphi)$, and $d'=\dim \mathrm{Exc}(\varphi^+)$.
\end{defi}

Note that if a flip is of type $(1,d')$ for some $d'$, then $d'=\dim X-2$ by \cite[Lemma 5-1-17]{KMM87}.

\medskip

The notion of \emph{diminished base locus} plays an important role in proving \thref{ithm}. We first recall the definition of diminished base locus here.

\begin{defi} \thlabel{base locus}
Let $X$ be a normal projective variety, and $D$ a pseudo-effective $\Q$-divisor on $X$.
\begin{itemize}
    \item[(a)] The \emph{stable base locus} of $D$ is
    $$ \B(D)=\bigcap_{m\in \Z_{>0}}\Bs(|mD|).$$
    \item[(b)] The \emph{diminished base locus} of $D$ is
    $$ \B_-(D):=\bigcup_{m\in \Z_{>0}} \B\left(D+\frac{1}{m}A\right),$$
    where $A$ is an ample divisor on $X$. Note that the definition does not depend on the choice of $A$.
\end{itemize}
\end{defi}

For more details on the diminished base locus, see \cite{ELM+06}.

Note that the diminished base locus is not always a closed subset, and there are no well-defined notions on the dimension of the diminished base locus (cf. \cite{Les14}).

Let us prove the following two lemmas on the notion of diminished base locus.

\begin{lem}\thlabel{diminished}
Let $X$ be a normal projective variety, and $D$ a pseudo-effective $\Q$-divisor on $X$. Then we have
$$ \B_-(D)=\bigcup_{m\in \Z_{>0}} \B_-\left(D+\frac{1}{m}A\right),$$
where $A$ is an ample divisor on $X$.
\end{lem}

\begin{proof}
$$ \B_-(D)=\bigcup_m\B\left(D+\frac{1}{m}A\right)\subseteq \bigcup_m\B_-\left(D+\frac{1}{2m}A\right)\subseteq \B_-(D),$$
and we obtain the assertion.
\end{proof}

\begin{lem}\thlabel{diminished2}
Let $(X,\Delta)$ be a projective klt pair, and assume that $K_X+\Delta$ is pseudo-effective. Let
$$ (X_0,\Delta_0):=(X,\Delta)\overset{f_1}{\dashrightarrow} (X_1,\Delta_1)\overset{f_2}{\dashrightarrow} \cdots$$
be a $(K_X+\Delta)$-MMP with scaling of an ample divisor. If all steps are flips of type $(1,\dim X-2)$, then any irreducible component of $\B_-(K_{X_i}+\Delta_i)$ has dimension $\le 1$.
\end{lem}

\begin{proof}
Let us assume that for some $i$, there is a dimension $\ge 2$ irreducible component 
$$Z\subseteq \B_-(K_{X_i}+\Delta_i).$$ Then by \cite[Lemma 2.6]{Mor25}, for some $j\ge i$, $Z$ is in the exceptional locus of $f_j\circ \cdots \circ f_i$. Thus, $f_{j'}$ is of type $(d',d'')$ for some $i\le j'\le j$, and for some $d'\ge 2$. This is a contradiction to the assumption.
\end{proof}

We end this section by proving the following lemma.

\begin{lem}\thlabel{KRss}
Let $X$ be a normal projective variety, $D$ a $\Q$-Cartier divisor on $X$ and $f:X\dashrightarrow X^+$ a $D$-flip. Suppose
$$ 
\begin{tikzcd}
X\ar[dashed,"f"]{rr}\ar["\varphi"']{rd}& & X^{+} \ar["\varphi^+"]{ld}\\
& Z&
\end{tikzcd}
$$
is the flipping diagram. Then there is a common resolution $g:X'\to X$ and $g':X'\to X^+$ such that the diagram
$$
\begin{tikzcd}
 & X'\ar["g"']{ld}\ar["g'"]{rd}&\\
    X\ar[dashed,"f"]{rr}\ar["\varphi"']{rd}& & X^{+} \ar["\varphi^+"]{ld}\\
& Z&
\end{tikzcd}
$$
is commutative.
\end{lem}

\begin{proof}
Note that any projective birational morphism is a blowup (cf. \cite[Theorem 8.1.24]{Liu02}). Thus, by the universal property of blowup (cf. \cite[Lemma 0806]{Stacks}), there is a proper birational morphism $g:X'\to X$ such that it also defines $g':X'\to X^+$, and $g,g'$ makes the commutative diagram. We may assume that $X'$ is smooth by the resolution of singularities (cf. \cite{Hir64}).
\end{proof}

\section{Valuation}
The aim of this section is to prove \thref{JM}, that is, a generalization of \cite[Theorem 7.3]{JM12}. We believe that \thref{JM} is well known by experts, but we can not find any reference to the proof, thus we include the proof, and it is the aim of this section. In this section, we work with a pair $(X,\Delta)$.

\subsection{Log canonical threshold and asymptotic multiplier ideals}

Let $\mathfrak{a}_{\bullet}:=\{\mathfrak{a}_m\}_{m \in \mathbb{Z}_{>0}}$ be a set of ideals in $\mathcal{O}_X$. We say that $\mathfrak{a}_{\bullet}$ is a \emph{graded sequence of ideals} in $\O_X$ if there exists a subsemigroup $\Phi \subseteq \mathbb{Z}_{>0}$ such that
$$
  \mathfrak{a}_{m} \cdot \mathfrak{a}_{m'} \subseteq \mathfrak{a}_{m+m'}
$$
for all $m,m'\in \Phi$. Let us define $\Phi(\mathfrak{a}_{\bullet})$ for a graded sequence of ideals $\mathfrak{a}_{\bullet}$ in $\O_X$ by
$$ \Phi(\mathfrak{a}_{\bullet}):=\{m\in \Z_{>0}\mid \mathfrak{a}_m\ne (0)\}.$$

Let $\mathfrak{a},\mathfrak{q}$ be ideals in $\O_X$, and let $\lambda>0$. Suppose $f:X' \to X$ is a log resolution of both $(X,\Delta)$ and $\mathfrak{a}\cdot\mathfrak{q}$. Let $F$ be an effective divisor on $X'$ such that
$$
  \mathfrak{a}\cdot\O_{X'}=\O_{X'}(-F).
$$
We define the \emph{multiplier ideal sheaf} $\mathcal{J}\left(X,\Delta;\mathfrak{a}^\lambda\right)$ by
$$
  \mathcal{J}\left(X,\Delta;\mathfrak{a}^\lambda\right):=
  f_*\O_{X'}\left(K_{X'}-\floor{f^*(K_X+\Delta)+\lambda F}\right).
$$
Note that this definition does not depend on the choice of $f$.

Let $\lambda>0$. We define the \emph{asymptotic multiplier ideal sheaf} $\mathcal{J}(X,\Delta;\mathfrak{a}^{\lambda}_{\bullet})$ as the largest (with respect to inclusion) among the sheaves
$$
  \left\{
    \mathcal{J}\left(X,\Delta;\mathfrak{a}_m^{\frac{\lambda}{m}}\right)\Bigm|
    m\in\Phi(\mathfrak{a}_{\bullet})
  \right\}.
$$
By the Noetherian property of $\O_X$, this maximum indeed exists.

For a graded sequence of ideals $\mathfrak{a}_{\bullet}$ in $\O_X$ and ideals $\mathfrak{a},\mathfrak{q}$ in $\O_X$, we define the \emph{log canonical threshold} as follows:
$$
  \lct^{\mathfrak{q}}(X,\Delta;\mathfrak{a}):=\inf \left\{\lambda\Bigm|\mathfrak{q} \subseteq\mathcal{J}\left(X,\Delta;\mathfrak{a}^\lambda_{\bullet}\right)
  \right\},
$$
$$
  \lct^{\mathfrak{q}}(X,\Delta;\mathfrak{a}_{\bullet}):=\sup_{m\ge 1}
  \left(m\cdot \lct^{\mathfrak{q}}\left(X,\Delta;\mathfrak{a}_m\right)\right).
$$

\begin{thm}\thlabel{-1}
Let $(X,\Delta)$ be a pair, $\mathfrak{q}$ an ideal in $\O_X$, and $\mathfrak{a}_{\bullet}$ a graded sequence of ideals in $\O_X$. Then
$$
  \lct^{\mathfrak{q}}(X,\Delta;\mathfrak{a}_{\bullet})=\inf \left\{\lambda\bigm|\mathfrak{q} \nsubseteq\mathcal{J}\left(X,\Delta;\mathfrak{a}_{\bullet}^{\lambda}\right)\right\}.
$$
\end{thm}

\begin{proof}
First, we show
$$
  \lct^{\mathfrak{q}}(X,\Delta;\mathfrak{a}_{\bullet})
  \le
  \inf \left\{\lambda\Bigm|\mathfrak{q} \nsubseteq\mathcal{J}\left(X,\Delta;\mathfrak{a}_{\bullet}^{\lambda}\right)
  \right\}.
$$
Let $\lambda$ be a positive number such that
$\lct^{\mathfrak{q}}(X,\Delta;\mathfrak{a}_{\bullet})>\lambda$.  
Then, for some sufficiently large $m_0$, we have
$\lct^{\mathfrak{q}}(X,\Delta;\mathfrak{a}_{m_0})>\frac{\lambda}{m_0}$.  
By the definition of the log canonical threshold, for any $\lambda'\le\lambda$, it follows that
$$
\mathfrak{q}\subseteq\mathcal{J}\left(X,\Delta;\mathfrak{a}_{\bullet}^{\lambda'}\right)=\mathcal{J}\left(X,\Delta;\mathfrak{a}_{m_0}^{\tfrac{\lambda'}{m_0}}\right).
$$
This proves the desired inequality.

For the converse, we note that for any positive integer $m$,
$$
  \lct^{\mathfrak{q}}(X,\Delta;\mathfrak{a}_{m})\le\frac{\lct^{\mathfrak{q}}(X,\Delta;\mathfrak{a}_{\bullet})}{m}.
$$
Hence, again by the definition of the log canonical threshold,
$$
  \mathfrak{q}\nsubseteq\mathcal{J}\left(X,\Delta;\mathfrak{a}_{m}^{\tfrac{\lct^{\mathfrak{q}}(X,\Delta;\mathfrak{a}_{\bullet})}{m}}
  \right).
$$
Therefore, by the definition of the asymptotic multiplier ideal, we also have
$$
\mathfrak{q}\nsubseteq\mathcal{J}\left(X,\Delta;\mathfrak{a}_{\bullet}^{\lct^{\mathfrak{q}}(X,\Delta;\mathfrak{a}_{\bullet})}\right),
$$
which completes the proof of the reverse inequality.
\end{proof}

Recalling the definition of the log canonical threshold, one obtains the following characterization (cf. \cite[Lemma 1.7]{JM12}):

\begin{prop}\thlabel{0}
Let $(X,\Delta)$ be a pair, and let $\mathfrak{a}, \mathfrak{q}$ be ideals in $\O_X$. Suppose $f : X' \to X$ is a log resolution of both $(X,\Delta)$ and $\mathfrak{a}\cdot \mathfrak{q}$. Then
$$
  \lct^{\mathfrak{q}}(X,\Delta;\mathfrak{a})=\inf_{E}
\frac{A_{X,\Delta}(E)+\ord_E(\mathfrak{q})}{\ord_E(\mathfrak{a})},
$$
where $E$ runs through all prime divisors in
$$
\mathrm{Supp}\left(f^{-1}_*\Delta\right)\cup\mathrm{Supp}\left((\mathfrak{a}\cdot \mathfrak{q})\cdot\O_{X'}\right).
$$
\end{prop}

\begin{proof}
Let $\lambda>0$. Take an effective divisor $F$ on $X'$ such that $\mathfrak{a}\cdot \O_{X'}=\O_{X'}(-F)$. Let $E_1,\cdots,E_r$ be the prime divisors in
$$
\mathrm{Supp}\left(f^{-1}_*\Delta\right)\cup\mathrm{Supp}\left((\mathfrak{a}\cdot \mathfrak{q})\cdot\O_{X'}\right).
$$
Observe that
$$ \mathfrak{q}\subseteq \mathcal{J}(X,\Delta;\mathfrak{a}^{\lambda}):=f_*\O_{X'}(K_{X'}-\floor{f^*(K_X+\Delta)+\lambda F})$$
holds if and only if 
$$\ord_{E_i}(K_{X'}-\floor{f^*(K_X+\Delta)+\lambda F})\ge -\ord_{E_i}(\mathfrak{q})\text{ for all }i.$$
This condition is equivalent to
$$A_{X,\Delta}(E_i)-\lambda \cdot \ord_{E_i}(\mathfrak{a})>-\ord_{E_i}(\mathfrak{q})\text{ for all }i.$$
which can be rewritten as
$$ \lambda<\frac{A_{X,\Delta}(E_i)+\ord_{E_i}(\mathfrak{q})}{\ord_{E_i}(\mathfrak{a})}\text{ for all }i.$$
Hence, by the definition of the log canonical threshold, $\mathrm{lct}^{\mathfrak{q}}(X,\Delta;\mathfrak{a})$ is the infimum of
$\frac{A_{X,\Delta}(E_i) + \ord_{E_i}(\mathfrak{q})}{\ord_{E_i}(\mathfrak{a})}$, completing the proof.
\end{proof}

\subsection{Valuation and log canonical threshold}
The goal of this subsection is to prove \thref{4}. We begin by recalling some basic notions on valuations; for further details, see \cite{JM12} and \cite[Chapter~1]{Xu}.

A \emph{valuation} $\nu$ over $X$ is a function
$$
  \nu : K(X) \to \mathbb{R}
$$
satisfying the following properties:
\begin{itemize}
    \item[(1)] $\nu(a)=0$ for all constants $a \in \mathbb{C}$.
    \item[(2)] $\nu(fg)=\nu(f)+\nu(g)$ for all $f,g \in K(X)$.
    \item[(3)] $\nu(f+g)\ge\min\{\nu(f),\nu(g)\}$ for all $f,g \in K(X)$.
\end{itemize}
We denote by $\mathrm{Val}_X$ the set of all valuations over $X$, and denote by $\mathrm{Val}^*_X$ the set of non-zero valuations over $X$. For any ideal $\mathfrak{a}$ in $\O_X$, define
$$
\nu(\mathfrak{a}):=\min \{\nu(f)\mid f \in \mathfrak{a}\}.
$$
We endow $\mathrm{Val}_X$ with the weakest topology such that, for every ideal $\mathfrak{a}$ in $\O_X$, the map
$$
  \nu\mapsto\nu(\mathfrak{a})
$$
is continuous.

Next, let $f:X'\to X$ be a resolution of $X$, and let $E \subseteq X'$ be a simple normal crossing (snc) divisor on $X'$. We say that $(X',E)$ is a \emph{log smooth model} over $X$. Write $E=E_1+\cdots+E_r$ as the sum of prime divisors on $X'$. For a subset $\{i_1,\dots,i_k\} \subseteq \{1,\dots,r\}$ with $1\le i_1<\cdots<i_k\le r$, let $
  \eta:=\eta_{i_1,\dots,i_k}$
be the generic point of $E_{i_1}\cap\cdots\cap E_{i_k}$ (if it exists). Locally around $\eta$, let $z_j$ be a local defining equation of $E_{i_j}$.

By the Cohen structure theorem (cf. \cite[Lemma 032D]{Stacks}), there is an isomorphism
$$
  \widehat{\O_{X',\eta}}\cong K(\eta)\left[\left[z_1,\dots,z_k\right]\right],
$$
and thus any $f\in\O_{X',\eta}$ can be expressed as
$$
  f=\sum_{\beta=(\beta_1,\dots,\beta_k)\in\mathbb{Z}_{\ge 0}^k}a_{\beta}z_1^{\beta_1}\cdots z_k^{\beta_k},
$$
where $a_{\beta} \in K(\eta)$. For any $\alpha=(\alpha_1,\dots,\alpha_k) \in \mathbb{R}_{\ge 0}^k$, we define
$$
  \nu_{(X',E),\alpha}(f):=\min\!\left\{\sum^k_{j=1}\alpha_j\beta_j\Bigm|a_{\beta}\neq0\right\}.
$$
It is straightforward to verify that $\nu_{(X',E),\alpha}$ is indeed a valuation over $X$. We denote by $\mathrm{QM}_{\eta}(X',E)$ the set of all valuations of the form $\nu_{(X',E),\alpha}$. Observe that we can naturally identify $\mathrm{QM}_{\eta}(X',E)$ with the Euclidean space $\mathbb{R}_{\ge 0}^k$. Define
$$
  \mathrm{QM}(X',E):=\bigcup_{1 \le i_1 < \cdots < i_k \le r}
  \mathrm{QM}_{\eta_{i_1,\dots,i_k}}(X',E).
$$
If $E_{i_1}\cap \cdots \cap E_{i_k}=\varnothing$, we set $\mathrm{QM}_{\eta}(X',E)=\varnothing$. For a valuation $\nu\in \mathrm{Val}_X$, we say that $\nu$ is \emph{quasi-monomial} if there is a log smooth model $(X',E)$ over $X$ such that $\nu \in \mathrm{QM}(X',E)$ holds.

We now introduce the notions of \emph{retraction map} and \emph{log discrepancy} for valuations. For any log smooth model $(X',E)$ over $X$, define the \emph{retraction map}
$$
  r_{(X',E)}:\mathrm{Val}_X \to \mathrm{QM}_{\eta}(X',E)
  \;\;\text{ by }\;\;
  \nu \mapsto \nu_{(X',E),(\nu(E_1),\dots,\nu(E_k))}.
$$
The above map also defines
$$
  r_{(X',E)}:\mathrm{Val}_X \to \mathrm{QM}(X',E).
$$
Next, we define
$$
  A_{X,\Delta}(\nu_{(X',E),\alpha}):=\sum_{j=1}^k \alpha_j A_{X,\Delta}(E_j),
$$
and for any $\nu \in \mathrm{Val}_X$,
$$
  A_{X,\Delta}(\nu):=\sup_{\substack{(X',E) \\ \text{log smooth over }X}}
  A_{X,\Delta}\left(r_{(X',E)}(\nu)\right).
$$
By construction, the map $A_{X,\Delta} : \mathrm{Val}_X \to \mathbb{R}$ is lower semi-continuous. For more details on the log discrepancy of valuations, see \cite[Chapter~5]{JM12} and \cite[Chapter~7]{Blu18}.

Let us prove the following lemma.

\begin{lem}\thlabel{1'}
Let $(X,\Delta)$ be a pair, $\mathfrak{a}$ an ideal in $\O_X$, and $f:X' \to X$ a log resolution of $\mathfrak{a}$. Let $\left(X',E:=\sum_{i=1}^r E_i\right)$ be a log smooth model over $X$, and let $\eta$ be the generic point of $E_1 \cap \cdots \cap E_r$. Then the map
$$
  \nu \mapsto \nu(\mathfrak{a})
$$
is a linear function on $\mathrm{QM}_{\eta}(X',E)$.
\end{lem}

\begin{proof}
Let $z_i$ be a local defining equation of $E_i$ around $\eta$. Since the divisor $\mathfrak{a} \cdot \O_{X'}$ is simple normal crossing on $X'$, it is principal in $\O_{X'}$ with a generator of the form
$$
  z^{\ord_{E_1}(\mathfrak{a})}_{1}\cdots z^{\ord_{E_r}(\mathfrak{a})}_{r}.
$$
Hence,
$$
  \nu(\mathfrak{a})=\nu\left(z^{\ord_{E_1}(\mathfrak{a})}_1 \cdots z^{\ord_{E_r}(\mathfrak{a})}_r\right),
$$
and therefore, by the definition of quasi-monomial valuations,
$$
  \nu_{(X',E),\alpha}(\mathfrak{a})=\sum_{i=1}^r\alpha_i\ord_{E_i}(\mathfrak{a}).
$$
This clearly shows that $\nu(\mathfrak{a})$ is linear in $\alpha$, so $\nu \mapsto \nu(\mathfrak{a})$ is a linear function on $\mathrm{QM}_{\eta}(X',E)$.
\end{proof}

\thref{1'} implies the following lemma.

\begin{lem}[{cf. \cite[Lemma 6.7]{JM12}}]\thlabel{2'}
Let $(X,\Delta)$ be a pair, and let $\mathfrak{a},\mathfrak{q}$ ideals in $\O_X$. Then
$$ \lct^{\mathfrak{q}}(X,\Delta;\mathfrak{a})=\inf_{\nu\in \mathrm{Val}^*_X}\frac{A_{X,\Delta}(\nu)+\nu(\mathfrak{q})}{\nu(\mathfrak{a})}.$$
\end{lem}

\begin{proof}
For simplicity, let us denote by
$$ L:=\inf_{\nu\in \mathrm{Val}^*_X}\frac{A_{X,\Delta}(\nu)+\nu(\mathfrak{q})}{\nu(\mathfrak{a})},$$
and let $E_1,\cdots,E_r$ be distinct prime divisors on $X'$ such that $E_1\cup \cdots \cup E_r=\Supp f^{-1}_*\Delta\cup \Supp (\mathfrak{a}\cdot \mathfrak{q})\cdot\O_{X'}$. Let $E:=\sum^r_{i=1} E$.

Let $f:X'\to X$ be a log resolution of both $(X,\Delta)$ and $\mathfrak{a}\cdot \mathfrak{q}$, and let us define
$$ \chi(\nu):=\frac{A_{X,\Delta}(\nu)+\nu(\mathfrak{q})}{\nu(\mathfrak{a})}\text{ for any }\nu\in \mathrm{Val}^*_X.$$
Then by \cite[Corollary 4.8, Corollary 5.3]{JM12},
$$ \left(\chi\circ r_{(X',E)}\right)(\nu)\le \chi(\nu)\text{ for any }\nu\in \mathrm{Val}_X.$$
Therefore,
$$ L=\inf_{\nu\in \mathrm{Val}^{*}_X}\chi(\nu)\ge \inf_{\nu\in \mathrm{QM}(Y,E)\setminus \{0\}}\chi(\nu)\ge \inf_{\nu\in \mathrm{Val}^{*}_X}\chi(\nu)=L,$$
and
$$L=\inf_{\nu\in \mathrm{QM}(X',E)\setminus \{0\}}\chi(\nu).$$
Let $\mathcal{D}(E)\subseteq \mathrm{QM}(X',E)$ be the set of those valuations in $\mathrm{QM}(X',E)$ 
satisfying $A_{X,\Delta}(\nu) = 1$. Then
$$L=\inf_{\nu\in \mathcal{D}(E)}\chi(\nu).$$
Moreover, there exists a subset $\{i_1,\cdots,i_k\}\subseteq \{1,\cdots,r\}$ with $1\le i_1<\cdots<i_k\le r$ and the generic point $\eta$ of $E_{i_1}\cap \cdots \cap E_{i_k}$ such that
$$L=\inf_{\nu\in \mathcal{D}(E)\cap \mathrm{QM}_{\eta}(X',E)}\chi(\nu).$$
Since $\chi:\mathcal{D}(E)\cap \mathrm{QM}_{\eta}(X',E)\to \R$ is lower semi-continuous and $\mathcal{D}(E)\cap \mathrm{QM}_{\eta}(X',E)$ is compact in the subspace topology
$$ \mathcal{D}(E)\cap \mathrm{QM}_{\eta}(X',E)\subseteq \mathrm{Val}_X,$$
there is some $\alpha=(\alpha_1,\cdots,\alpha_k)\in \R^k_{\ge 0}$ such that
\begin{equation}\label{M}
L=\chi(\nu_{(X',E),\alpha})
\end{equation}

Define
$$ \phi(\nu):=A_{X,\Delta}(\nu)+\nu(\mathfrak{q})-L\cdot \nu(\mathfrak{a}).$$
By \thref{1'},
$$ \phi:\mathrm{QM}_{\eta}(X',E)\to \R$$
is a linear function.
\begin{itemize}
    \item[(i)] We know $\phi(\nu)\ge 0$ for any $\nu\in \mathrm{QM}_{\eta}(X',E)$, and
    \item[(ii)] We also note that $L=\chi(\nu)$ if and only if $\nu=0$ or $\phi(\nu)=0$.
\end{itemize}
Thus, by combining (i) and (ii) with (\ref{M}), we know that for any $j$ with $\alpha_j\ne 0$, $L=\chi\left(\ord_{E_{i_j}}\right)$. Hence, by \thref{0}, the assertion follows.
\end{proof}

Let $\mathfrak{a}_{\bullet}$ be a graded sequence of ideals in $\O_X$. For a positive integer $m$, define
$$\mathfrak{b}_m:=\mathcal{J}(X,\Delta;\mathfrak{a}^m_{\bullet}).$$
Moreover, let us define
$$ \nu\left(\mathfrak{a}_{\bullet}\right):=\lim_{\substack{m\to \infty\\m\in \Phi(\mathfrak{a}_{\bullet})}}\frac{\nu\left(\mathfrak{a}_m\right)}{m}.$$
By \cite[Lemma 2.3]{JM12}, the limit exists.

Recall the following proposition from \cite{Xu}.

\begin{prop}[{cf. \cite[Lemma 1.58]{Xu}}]\thlabel{cor:tak}
Let $(X,\Delta)$ be a klt pair and $\mathfrak{a}_{\bullet}$ a graded sequence of ideals of $\O_X$. There is a nonzero ideal $I$ in $\O_X$ which depends only on $(X,\Delta)$ such that for any $m,m'\in \Phi(\mathfrak{a}_{\bullet})$, we have $I\cdot \mathfrak{b}_{m+m'}\subseteq \mathfrak{b}_m\cdot \mathfrak{b}_{m'}$.
\end{prop}

\begin{lem}[{cf. \cite[Lemma 1.59 (ii)]{Xu}}] \thlabel{3}
Let $(X,\Delta)$ be a klt pair, $\mathfrak{q}$ an ideal in $\O_X$, and $\mathfrak{a}_{\bullet}$ a graded sequence of ideals on $X$. Then we have
$$\lim_{m\to \infty}m\cdot \lct^{\mathfrak{q}}\left(X,\Delta;\mathfrak{b}_m\right)=\lct^{\mathfrak{q}}(X,\Delta;\mathfrak{a}_{\bullet}).$$
\end{lem}

\begin{proof}
If $\lct^{\mathfrak{q}}(X,\Delta;\mathfrak{a}_{\bullet})=\infty$, then the claimed equality holds. Indeed, for any positive integer $m$, since $\mathfrak{a}_m\subseteq \mathfrak{b}_m$, we have 
\begin{equation} \label{J}
\lct^{\mathfrak{q}}\left(X,\Delta;\mathfrak{b}_m\right)\ge \lct^{\mathfrak{q}}\left(X,\Delta;\mathfrak{a}_m\right).
\end{equation}
Hence we may assume that $\lct^{\mathfrak{q}}(X,\Delta;\mathfrak{a}_{\bullet})<\infty$. Suppose $p$ is a positive integer. By \thref{0}, there is a prime divisor $E$ over $X$ such that
$$
  \lct^{\mathfrak{q}}(X,\Delta;\mathfrak{a}_{mp})=\frac{A_{X,\Delta}(E)+\nu(\mathfrak{q})}{\nu\left(\mathfrak{a}_{mp}\right)}.
$$
By \thref{0} again,
$$
  m \cdot
  \lct^{\mathfrak{q}}(X,\Delta;\mathfrak{b}_m)\le\frac{A_{X,\Delta}(E) + \ord_E(\mathfrak{q})}{\tfrac{1}{m}\ord_E(\mathfrak{b}_m)},
$$
and by definition of $\mathfrak{b}_m$,
$$
  \frac{1}{m}\ord_E(\mathfrak{b}_m)\ge\frac{1}{mp}\ord_E\left(\mathfrak{a}_{mp}\right)-
  \frac{1}{m}\left(A_{X,\Delta}(E)+\ord_E(\mathfrak{q})\right).$$
for any positive integer $p$. Hence
$$
m\cdot\lct^{\mathfrak{q}}(X,\Delta;\mathfrak{b}_m)\le\frac{A_{X,\Delta}(E)+\ord_E(\mathfrak{q})}{\tfrac{1}{mp}\ord_E(\mathfrak{a}_{mp})-\tfrac{1}{m}\left(A_{X,\Delta}(E)+\ord_E(\mathfrak{q})\right)}=\frac{mp\cdot\lct^{\mathfrak{q}}(X,\Delta;\mathfrak{a}_{mp})}{1-p\cdot\lct^{\mathfrak{q}}(X,\Delta;\mathfrak{a}_{mp})}.
$$
Letting $p \to \infty$ and $m\to \infty$ successively, we conclude
$$
  \lim_{m \to \infty}
  \left(m \cdot
    \lct^{\mathfrak{q}}(X,\Delta;\mathfrak{b}_m)
  \right)\le\lct^{\mathfrak{q}}(X,\Delta;\mathfrak{a}_\bullet),
$$
and the reverse inequality is derived from (\ref{J}). This completes the proof.
\end{proof}

We prove the following theorem.

\begin{thm}\thlabel{4}
Let $(X,\Delta)$ be a klt pair, $\mathfrak{q}$ an ideal in $\O_X$, and $\mathfrak{a}_{\bullet}$ a graded sequence of ideals in $\O_X$. Then
$$
  \lct^{\mathfrak{q}}(X,\Delta;\mathfrak{a}_{\bullet})
  =
  \inf_{\nu \in \mathrm{Val}_X^*}
  \frac{A_{X,\Delta}(\nu)+\nu(\mathfrak{q})}{\nu(\mathfrak{a}_{\bullet})}.
$$
\end{thm}

\begin{proof}
We argue similarly to the proof of \cite[Lemma~1.60]{Xu}.

First, observe that if
$\lct^{\mathfrak{q}}(X,\Delta;\mathfrak{a}_{\bullet}) = \infty$,
the statement is trivial. Hence we may assume
$\lct^{\mathfrak{q}}(X,\Delta;\mathfrak{a}_{\bullet}) < \infty$.

By \thref{cor:tak}, for any positive integer $n$, there exists an ideal $I$ in $\O_X$ such that
$$
  I^n \cdot \mathfrak{b}_{(n+1)!}\subseteq\left(\mathfrak{b}_{n!}\right)^{n+1}.
$$
Then, by \cite[Lemma~1.59 (i)]{Xu}, we have the chain of inequalities
$$
\begin{aligned}
  \frac{1}{n!}\ord_E\left(\mathfrak{b}_{n!}\right)&\le\frac{n}{(n+1)!}\ord_E(I)+\frac{1}{(n+1)!}\ord_E\left(\mathfrak{b}_{(n+1)!}\right)
  \\ &\le\sum_{n' \ge n}\frac{n'}{(n'+1)!}\ord_E(I)+\ord_E\left(\mathfrak{a}_{\bullet}\right)
  \\ &=\frac{1}{n!}\ord_E(I)+\ord_E\left(\mathfrak{a}_{\bullet}\right).
\end{aligned}
$$

Fix a constant $C$ such that
$\ord_E(I) \le C \cdot A_{X,\Delta}(E)$ for every prime divisor $E$ over $X$. By \thref{3}, for any $\varepsilon>0$, there exists a positive integer $n$ satisfying $\tfrac{C}{n!} < \tfrac{\varepsilon}{2}$ and
$$
  \left|
    \frac{1}{n!\lct^{\mathfrak{q}}(X,\Delta;\mathfrak{b}_{n!})}-\frac{1}{\lct^{\mathfrak{q}}(X,\Delta;\mathfrak{a}_{\bullet})}
  \right|\le\frac{\varepsilon}{2}.
$$
Let $E$ be a prime divisor over $X$ such that
$$
  \lct^{\mathfrak{q}}(X,\Delta;\mathfrak{b}_{n!})=\frac{A_{X,\Delta}(E)+\ord_E(\mathfrak{q})}
       {\ord_E\left(\mathfrak{b}_{n!}\right)}.
$$
Then
$$
  \frac{1}{n!\lct^{\mathfrak{q}}(X,\Delta;\mathfrak{b}_{n!})}=\frac{\ord_E\left(\mathfrak{b}_{n}\right)}{n!A_{X,\Delta}(E) + n!\ord_E(\mathfrak{q})}\le\frac{\ord_E\left(\mathfrak{a}_{\bullet}\right)}{A_{X,\Delta}(E) + \ord_E(\mathfrak{q})}+\frac{\varepsilon}{2}.
$$
Hence we obtain
$$
\begin{aligned}
  0&\le
  \frac{1}{\lct^{\mathfrak{q}}(X,\Delta;\mathfrak{a}_{\bullet})}-\frac{
    \ord_E\left(\mathfrak{a}_{\bullet}\right)}{A_{X,\Delta}(E) + \ord_E(\mathfrak{q})}
    \\ &=\left(
    \frac{1}{\lct^{\mathfrak{q}}(X,\Delta;\mathfrak{a}_{\bullet})}-\frac{1}{n!\lct^{\mathfrak{q}}(X,\Delta;\mathfrak{b}_{n!})}\right)+\left(
    \frac{1}{n!\lct^{\mathfrak{q}}(X,\Delta;\mathfrak{b}_{n!})}-\frac{\ord_E\left(\mathfrak{a}_{\bullet}\right)}{A_{X,\Delta}(E)+\ord_E(\mathfrak{q})}
  \right)
  \\ &\le\varepsilon.
\end{aligned}
$$
Since $\varepsilon > 0$ is arbitrary, we conclude
$$
  \lct^{\mathfrak{q}}(X,\Delta;\mathfrak{a}_{\bullet})\ge\inf_{\nu \in\mathrm{Val}_X^*}
  \frac{A_{X,\Delta}(\nu) + \nu(\mathfrak{q})}{\nu\left(\mathfrak{a}_{\bullet}\right)}.
$$

On the other hand, for every $m \in \Phi(\mathfrak{a}_{\bullet})$ and any $\nu \in \mathrm{Val}_X^*$, we have
$$
  m \cdot
  \lct^{\mathfrak{q}}(X,\Delta;\mathfrak{a}_m)
  \overset{(1)}{\le}
  \frac{A_{X,\Delta}(\nu) + \nu(\mathfrak{q})}{\tfrac{\nu(\mathfrak{a}_m)}{m}}\le\frac{A_{X,\Delta}(\nu) + \nu(\mathfrak{q})}{\nu\left(\mathfrak{a}_{\bullet}\right)},
$$
where (1) follows from \thref{2'}. This establishes the reverse inequality, completing the proof.
\end{proof}

\begin{defi}
Let $(X,\Delta)$ be a pair, $\mathfrak{q}$ an ideal in $\O_X$, and $\mathfrak{a}_{\bullet}$ a graded sequence of ideals in $\O_X$. For $\nu \in \mathrm{Val}_X$, we say that $\nu$ \emph{computes} $\lct^{\mathfrak{q}}(X,\Delta;\mathfrak{a}_{\bullet})$ if
$$
  \lct^{\mathfrak{q}}(X,\Delta;\mathfrak{a}_{\bullet})=\frac{A_{X,\Delta}(\nu) + \nu(\mathfrak{q})}{\nu\left(\mathfrak{a}_{\bullet}\right)}.
$$
\end{defi}

\subsection{Jonsson-Mustață theorem}

Recall the terminology of \cite{BdFF+15}. For any normal projective variety $X$ and a closed subscheme $N$ on $X$, $N$ is a \emph{normalizing subscheme} of $X$ if $N$ contains the singular locus of $X$. Moreover, we define
$$ \mathrm{Val}^N_X:=\{\nu\in \mathrm{Val}_X\mid\nu(\mathcal{I}_N)=1\},$$
and endow it with the subspace topology. The following proposition is almost the same as \cite[Proposition B.3]{Blu18}, and we include the statement for the convenience of the readers.

\begin{prop}\thlabel{p3}
Let $(X,\Delta)$ be a klt pair, $\mathfrak{a}_{\bullet}$ a graded sequence of ideals in $\O_X$, and $N$ a normalizing subscheme of $X$ such that $N$ contains the zero locus of $\mathfrak{a}_1$. Then the following hold.
\begin{itemize}
    \item[(a)] The function $\nu\mapsto \nu(\mathfrak{a}_{\bullet})$ is bounded on $\mathrm{Val}^N_X$.
    \item[(b)] For each $M\in \R$, the set $\mathrm{Val}^N_X\cap \{A_{X,\Delta}(\nu)\le M\}$ is compact.
    \item[(c)] There exists $\varepsilon>0$ such that $A_{X,\Delta}(\nu)>\varepsilon$ for all $\nu\in \mathrm{Val}^N_X$.
\end{itemize}
\end{prop}

\begin{proof}
(a) and (b) follow from \cite[Proposition 2.5 (a)]{BdFF+15}, and (c) can be proved almost the same as the proof of \cite[Proposition B.3 (3)]{Blu18}.
\end{proof}

We now prove a generalization of \cite[Theorem 7.3]{JM12}.

\begin{thm}\thlabel{JM}
Let $(X,\Delta)$ be a projective klt pair, $\mathfrak{q}$ an ideal sheaf in $\O_X$, and $\mathfrak{a}_{\bullet}$ a graded sequence of ideals in $\O_X$ such that $\lct^{\mathfrak{q}}(X,\Delta;\mathfrak{a}_{\bullet})<\infty$. Then there is $\nu\in \mathrm{Val}_X$ computing $\lct^{\mathfrak{q}}(X,\Delta;\mathfrak{a}_{\bullet})$.
\end{thm}

\begin{proof}
We note that the argument is nearly identical to that of \cite[Theorem B.1]{Blu18}.

Let $N$ be the subscheme of $X$ defined by $\mathcal{I}_{\mathrm{Sing}(X)}\cdot \mathfrak{a}_1$. By (a) and (c) in \thref{p3}, there are $M\in \R$ and $\varepsilon>0$ such that $\nu(\mathfrak{a}_{\bullet})<M$, and $A_{X,\Delta}(\nu)+\nu(\mathfrak{q})>\varepsilon$ for all $\nu\in \mathrm{Val}^N_X$. Moreover, by normalizing,
$$ \lct^{\mathfrak{q}}(X,\Delta;\mathfrak{a}_{\bullet})=\inf_{\nu\in \mathrm{Val}^N_X}\frac{A_{X,\Delta}(\nu)+\nu(\mathfrak{q})}{\nu(\mathfrak{a}_{\bullet})}.$$
Fix $L>\lct^{\mathfrak{q}}(X,\Delta;\mathfrak{a}_{\bullet})$.

Define
$$ W:=\mathrm{Val}^N_X\cap \{A_{X,\Delta}(\nu)\le ML\}\cap \{L\nu(\mathfrak{a}_{\bullet})\ge \varepsilon\}.$$
Using \cite[Lemma~6.1]{JM12} and (b) of \thref{p3}, one can show that $W$ is compact. For any $\nu\in \mathrm{Val}_X$ with $\frac{A_{X,\Delta}(\nu)+\nu(\mathfrak{q})}{\nu(\mathfrak{a}_{\bullet})}<L$, $A_{X,\Delta}(\nu)\le ML$ and $L\nu(\mathfrak{a}_{\bullet})\ge \varepsilon$. Thus,
$$ \lct^{\mathfrak{q}}(X,\Delta;\mathfrak{a}_{\bullet})=\inf_{\nu\in W}\frac{A_{X,\Delta}(\nu)+\nu(\mathfrak{q})}{\nu(\mathfrak{a}_{\bullet})}.$$

Since $\nu\mapsto A_{X,\Delta}(\nu)$ is lower semi-continuous, we obtain $\nu\mapsto \frac{A_{X,\Delta}(\nu)+\nu(\mathfrak{q})}{\nu(\mathfrak{a}_{\bullet})}$ is also lower semi-continuous as a function on $W$, and thus there exists $\nu\in W$ such that
$$ \lct^{\mathfrak{q}}(X,\Delta;\mathfrak{a}_{\bullet})=\frac{A_{X,\Delta}(\nu)+\nu(\mathfrak{q})}{\nu(\mathfrak{a}_{\bullet})}.$$
This completes the proof.
\end{proof}

\section{Diminished multiplier ideal and vanishing theorem} \label{Section 4}
The goal of this section is to define the notion of the \emph{diminished multiplier ideal} introduced by Hacon and Lehmann (cf. \cite{Hac04,Leh14}) for singular varieties, and to prove a corresponding vanishing theorem. The structure of this section closely follows that of \cite{Leh14}.

We begin by stating the definition of the \emph{asymptotic order} of a divisor.

\begin{defi}[{cf. \cite{ELM+06}, \cite[Defiition 2.2, and Chapter III, 1.6 Definition]{Nak04}}]
Let $X$ be a normal projective variety, $\nu\in \mathrm{Val}_X$ a valuation, $A$ an ample divisor on $X$, and $D$ a $\Q$-Cartier divisor on $X$ such that $\kappa(D)\ge 0$. For a positive integer $n$ such that $|nD|$ is non-empty, we define
$$ \nu(\|D\|):=\lim_{m\to \infty}\frac{\nu(\mathcal{I}_{\mathrm{Bs}(|mnD|)})}{nm}.$$
We then set
$$ \sigma_{\nu}(D):=\lim_{\substack{\varepsilon\to 0 \\ \varepsilon\text{ is rational}}}\nu(\|D+\varepsilon A\|).$$
Note that $\sigma_{\nu}(D)$ is independent to the choice of ample divisor (cf. \cite[1.5 in Chapter III]{Nak04}).
\end{defi}

Let us prove the following lemma.

\begin{lem}[{cf. \cite[Lemma 2.10]{Leh14}}]\thlabel{mag?}
Let $X$ be a normal projective variety, $A$ an ample divisor on $X$, and $D$ a pseudo-effective $\Q$-divisor on $X$. Let $\nu\in \mathrm{Val}_X$ be a valuation.
\begin{itemize}
    \item[(a)] The function
$$ f(t):=\nu(\|D+tA\|)$$
is strictly decreasing for any rational $t>0$.
\item[(b)] If $D$ is big, then $\nu(\|D\|)=\sigma_{\nu}(D)$.
\end{itemize}
\end{lem}

\begin{proof}
Let us prove (a). Suppose $0<t<t'$. Choose $0<\varepsilon\ll t'-t$. Since $A$ is ample, the divisor $(t'-t)A+\varepsilon(D+tA)$ is also an ample divisor. We obtain
$$ 
\begin{aligned}
\nu(\|D+t'A\|)&=\nu(\|(1-\varepsilon)(D+tA)+(t'-t)A+\varepsilon(D+tA)\|)
\\ &\le (1-\varepsilon)\nu(\|D+tA\|)<\nu(\|D+tA\|).
\end{aligned}
$$
Thus $\nu(\|D+tA\|)$ is strictly larger than $\nu(\|D+t'A\|)$, proving that $f(t)$ is strictly decreasing in $t$.

Let us prove (b). Assume $D$ is big. From (a), $\nu(\|D+tA\|)<\nu(\|D\|)$, so $\sigma_{\nu}(D)<\infty$.

Next, for sufficiently small $t>0$, the divisor $D-tA$ is also big. For big $D_1$ and $D_2$, we have
$$\sigma_{\nu}(D+D')\le \sigma_{\nu}(D)+\sigma_{\nu}(D').$$

We use this as follows. Write $t=t't''$ for rational $t',t''>0$. Then
$$ 
\begin{aligned}
\nu(\|D\|)&\le \sigma_{\nu}(D-tA)
\\ &\le (1-t')\sigma_{\nu}(D)+t'\sigma_{\nu}(D-t''A).
\end{aligned}$$
Letting $t'\to 0+$ gives
$$\nu(\|D\|)\le \lim_{\substack{t\to 0+ \\ t\text{ is rational}}}\sigma_{\nu}(D-tA)\le \sigma_{\nu}(D)\le \nu(\|D\|),$$
and hence $\sigma_{\nu}(D)=\nu(\|D\|)$. This completes the proof.
\end{proof}

Let us state the definition of asymptotic multiplier ideal.

\begin{defi}\thlabel{defi}
Let $(X,\Delta)$ be a projective pair.
\begin{itemize}
    \item[(a)] Let $D$ be an effective $\Q$-Cartier $\Q$-divisor on $X$, and let $f:X'\to X$ be a log resolution of $(X,\Delta)$. We define
    $$ \mathcal{J}(X,\Delta;D):=f_*\O_{X'}(K_{X'}-\floor{f^*(K_X+\Delta+D)}).$$
    \item[(b)] Under the same settings as in (a), assume that $D$ is Cartier. Let $F$ be the fixed part of $f^*D$. If $c>0$ is a rational number, then we define
    $$ \mathcal{J}(X,\Delta;c|D|):=f_*\O_{X'}(K_{X'}-\floor{f^*(K_X+\Delta)+cF}).$$
    \item[(c)] Let $D$ be an effective $\Q$-Cartier $\Q$-divisor on $X$. Suppose $n$ is a positive integer such that $nD$ is Cartier. We denote by $\mathcal{J}(X,\Delta;\|D\|)$ the maximal element among the set $\left\{\mathcal{J}(X,\Delta;\frac{1}{mn}|mnD|)\right\}_{m\in \Z_{>0}}$.
\end{itemize}
\end{defi}

\begin{rmk}\thlabel{rmk}
Using the same notation as in \emph{(b)} of \thref{defi}, assume $D$ is big. By \cite[Lemma 9.1.9]{Laz04}, if $f^*D\sim M+F$ for some basepoint-free divisor $M$ on $X'$, then there exists a divisor $D'\sim_{\Q}D$ such that
$$ \mathcal{J}(X,\Delta;c|D|)=\mathcal{J}(X,\Delta;cD').$$
A similar statement holds for asymptotic multiplier ideals.

One might also consider the case where $D$ is an $\R$-Cartier divisor. In \cite{Leh14}, the ideal $\mathcal{J}(\|D\|)$ is defined for any effective $\R$-Cartier divisor $D$ with $\kappa(D)\ge 0$. However, at present, we cannot directly define $\mathcal{J}(X,\Delta;\|D\|)$ under the same approach because $\floor{D}$ may fail to be Cartier.
\end{rmk}

\begin{coro}\thlabel{mag}
Let $(X,\Delta)$ be a projective klt pair, and let $D$ be a $\Q$-Cartier $\Q$-divisor on $X$ such that $\kappa(D)\ge 0$. Suppose $\mathfrak{q}$ is an ideal sheaf on $X$. Then the following conditions are equivalent:
\begin{itemize}
    \item[(a)] $\mathfrak{q}\subseteq \mathcal{J}(X,\Delta;\|D\|)$.
    \item[(b)] $\nu(\mathfrak{q})>\nu(\|D\|)-A_{X,\Delta}(\nu)$.
\end{itemize}
\end{coro}

\begin{proof}
Define a graded sequence of ideals $\mathfrak{a}_{\bullet}$ by
$$ \mathfrak{a}_m:=\begin{cases}\mathfrak{b}(|mD|)& \text{If }mD\text{ is Cartier,} \\
(0) & \text{Otherwise.}\end{cases}$$ 
Let $\nu'$ be a valuation of $X$ that computes $\lct^{\mathfrak{q}}(X,\Delta;\mathfrak{a}_{\bullet})$. Such a valuation exists by \thref{JM}. We then have:
$$
\begin{aligned}
\mathfrak{q}\subseteq \mathcal{J}(X,\Delta;\|D\|)&\overset{(1)}{\iff} \mathfrak{q}\subseteq \mathcal{J}(X,\Delta;\mathfrak{a}_{\bullet})
\\ &\overset{(2)}{\iff} \lct^{\mathfrak{q}}(X,\Delta;\mathfrak{a}_{\bullet})>1
\\ &\overset{(3)}{\iff} \frac{A_{X,\Delta}(\nu')+\nu'(\mathfrak{q})}{\nu'(\mathfrak{a}_{\bullet})}>1
\\ &\overset{(4)}{\iff} \frac{A_{X,\Delta}(\nu)+\nu(\mathfrak{q})}{\nu(\|D\|)}>1\text{ for every }\nu\in \mathrm{Val}^*_X,
\end{aligned}
$$
where $(1)$ is due to the definition of $\mathcal{J}(X,\Delta;\|D\|)$, $(2)$ follows from \thref{-1}, $(3)$ is the definition of $\nu'$, and $(4)$ again uses \thref{-1}.
\end{proof}

\begin{lem}\thlabel{KRs}
Let $(X,\Delta)$ be a projective klt pair, and $A$ an ample divisor on $X$. For any $0<\varepsilon'<\varepsilon$, we have
$$ \mathcal{J}(X,\Delta;\|D+\varepsilon' A\|)\subseteq \mathcal{J}(X,\Delta;\|D+\varepsilon A\|).$$
\end{lem}

\begin{proof}
This follows immediately from Corollary \ref{mag}.
\end{proof}

We will use the following Nadel vanishing theorem, which holds in the singular setting, to prove that the diminished multiplier ideal is well-defined.

\begin{thm}[{cf. \cite[Theorem 5.3]{CJK23}}]\thlabel{Nadel}
Let $(X,\Delta)$ be a projective klt pair, and let $D$ be a big $\Q$-divisor on $X$. Suppose that $L$ is a Cartier divisor on $X$ such that $L-(K_X+\Delta+D)$ is nef. Then
$$ H^i(X,\O_X(L)\otimes \mathcal{J}(X,\Delta;\|D\|))=0\text{ for }i>0.$$
\end{thm}

We now prove a generalization of \cite[Theorem 4.2]{Leh14} to the singular setting. The proof closely follows the original argument but includes adjustments for all pairs.

\begin{lem}
Let $(X,\Delta)$ be a projective pair, and $D_1,D_2,\cdots$ a sequence of big $\Q$-divisors on $X$. Suppose:
\begin{enumerate}
    \item[\emph{(a)}] There is a Cartier divisor $L$ on $X$ such that $L-D_m$ is ample for every $m$.
    \item[\emph{(b)}] There is a chain of inclusions
    $$ \cdots \subseteq \mathcal{J}(X,\Delta;\|D_2\|) \subseteq \mathcal{J}(X,\Delta;\|D_1\|).$$
\end{enumerate}
Then there exists a positive integer $m$ such that
$$ \mathcal{J}(X,\Delta;\|D_{m'}\|)=\mathcal{J}(X,\Delta;\|D_{m}\|)\,\,\text{ for all }\,\,m'\ge m.$$
\end{lem}

\begin{proof}
By (a), for each $m$, the divisor $L-(K_X+\Delta+D_m)$ is ample. Hence, By \thref{Nadel}, for any integer $n\ge 0$,
$$ H^i(X,\O_X(L+nH)\otimes \mathcal{J}(X,\Delta;\|D_m\|))=0\text{ for }i>0.$$
\medskip
By the Castelnuovo-Mumford regularity, the sheaf
$$ \O_X(L+nH)\otimes \mathcal{J}(X,\Delta;\|D_m\|)$$
is globally generated if $n\ge \dim X+1$. Fix an $n\ge \dim X+1$. Then the sheaf 
$$\O_X(L+nH)\otimes \mathcal{J}(X,\Delta;\|D_m\|)$$
is determined by a subspace of the finite-dimensional $\C$-vector space $H^0(X,\O_X(L+nH))$, namely $$H^0(X,\O_X(L+nH)\otimes \mathcal{J}(X,\Delta;\|D_m\|)).$$
This completes the proof.
\end{proof}

We can define a version of the multiplier ideal that is an analog of \cite[Definition 4.3]{Leh14}.

\begin{defi}
Let $(X,\Delta)$ be a projective klt pair, let $D$ be a pseudo-effective $\Q$-divisor on $X$, and let $A$ be an ample divisor on $X$. We define
$$ \mathcal{J}_{-}(X,\Delta;\|D\|):=\bigcap_{\varepsilon>0}\mathcal{J}(X,\Delta;\|D+\varepsilon A\|).$$
\end{defi}

\begin{prop}
Let $(X,\Delta)$ be a projectiev klt pair, and let $D$ be a pseudo-effective $\Q$-divisor on $X$. Then $\mathcal{J}_-(X,\Delta;\|D\|)$ is independent of the choice of an ample divisor $A$ on $X$.
\end{prop}

\begin{proof}
Let $A'$ be another ample divisor on $X$. Pick $m,m'\gg 0$, such that $mA'-A,m'A-mA'$ are both ample. By \thref{KRs}, we obtain the inclusions
$$
\mathcal{J}(X,\Delta;\|D+\varepsilon A\|)\subseteq \mathcal{J}(X,\Delta;\|D+m\varepsilon A'\|)\subseteq \mathcal{J}(X,\Delta;\|D+m'\varepsilon A\|)
$$
Taking the intersection over all $\varepsilon>0$ show the assertion.
\end{proof}
Let us prove the following:

\begin{prop}\thlabel{mag2}
Let $(X,\Delta)$ be a projective klt pair, and let $D$ be a pseudo-effective $\Q$-divisor on $X$. For any ideal $\mathfrak{q}$ in $\O_X$,
$$ \mathfrak{q}\subseteq\mathcal{J}_-(X,\Delta;\|D\|)\iff \nu(\mathfrak{q})\ge \sigma_{\nu}(D)-A_{X,\Delta}(\nu),$$
where the inequality is required to hold for every valuation $\nu\in \mathrm{Val}^*_X$.
\end{prop}

\begin{proof}
Let $A$ be an ample divisor on $X$. 

$(\Rightarrow)$ If $\mathfrak{q}\subseteq \mathcal{J}_-(X,\Delta;\|D\|)$, then by definition and \thref{mag}, 
$$\nu(\mathfrak{q})>\nu(\|D+\varepsilon A\|)-A_{X,\Delta}(\nu)$$
for all $\nu\in \mathrm{Val}^*_X$, and sufficiently small $\varepsilon>0$. Letting $\varepsilon \to 0$ and recalling the definition of $\sigma_{\nu}(D)$ completes the argument for the forward implication.

$(\Leftarrow)$ For the reversed implication, note that for every small $t>0$, $$\nu(\mathfrak{q})>\nu(\|D+tA\|)-A_{X,\Delta}(\nu).$$
By (b) of \thref{mag?}. Hence, invoking \thref{mag}, we see that $\mathfrak{q}\subseteq \mathcal{J}(X,\Delta;\|D\|)$.

Thus the two conditions are indeed equivalent.
\end{proof}

\begin{rmk}
Because $\mathcal{J}_{-}(X,\Delta;\|D\|)$ is a torsion-free sheaf on X, \thref{mag2} also implies that for any pseudo-effective $\Q$-divisor $D$ on $X$ and any valuation $\nu\in \mathrm{Val}_X$, $\sigma_{\nu}(D)<\infty.$
\end{rmk}

Let us state the following lemma.

\begin{lem}[{cf. \cite[Lemma 4.6]{CJK23}}]\thlabel{KR}
Let $(X,\Delta)$ be a projective klt pair, and let $D$ be a big $\Q$-divisor on $X$. Then for any rational $0<\varepsilon\ll 1$, we have
$$ \mathcal{J}(X,\Delta;\|D\|)=\mathcal{J}(X,\Delta;\|(1+\varepsilon)D\|).$$
\end{lem}

\begin{prop}\thlabel{minicoro}
Let $(X,\Delta)$ be a projective klt pair, and let $D$ be a pseudo-effective $\Q$-divisor on $X$. For any $0<\varepsilon''<\varepsilon<1$, the following inclusions hold:

\begin{itemize}
    \item[(a)] If $D$ is big, then
    $$ \mathcal{J}(X,\Delta;\|(1+\varepsilon)D\|)\subseteq \mathcal{J}(X,\Delta;\|(1+\varepsilon'')D\|).$$
    \item[(b)] If $D$ is pseudo-effective in general, then
    $$ \mathcal{J}_{-}(X,\Delta;\|(1+\varepsilon)D\|)\subseteq \mathcal{J}_{-}(X,\Delta;\|(1+\varepsilon'')D\|)$$
\end{itemize}
\end{prop}

\begin{proof}
The big case follows by applying \thref{KR}. 

For the second assertion, let $A$ be an ample divisor on $X$. If $0<\varepsilon'\ll \varepsilon, \varepsilon''$, then
$$ 
\begin{aligned}
\mathcal{J}_{-}(X,\Delta;\|(1+\varepsilon)D\|)&=\mathcal{J}(X,\Delta;\|(1+\varepsilon)(D+\varepsilon'A)\|)
\\ &\subseteq \mathcal{J}(X,\Delta;\|(1+\varepsilon'')(D+\varepsilon' A)\|)
\\ &=\mathcal{J}_-(X,\Delta;\|(1+\varepsilon'')D\|).
\end{aligned}$$
\end{proof}

Now, \thref{minicoro} gives us the definition of diminished multiplier ideal.

\begin{defi}\thlabel{realdef}
Let $(X,\Delta)$ be a projective klt pair, and let $D$ be a pseudo-effective $\Q$-divisor on $X$. We define the \emph{diminished multiplier ideal} by
$$ \mathcal{J}_{\sigma}(X,\Delta;\|D\|):=\bigcup_{\varepsilon>0}\mathcal{J}_{-}(X,\Delta;\|(1+\varepsilon)D\|).$$
\end{defi}

\begin{rmk}\thlabel{rmk}
Using the same notation as above, and choosing $0< \varepsilon' \ll \varepsilon \ll 1$, we also get
$$
\mathcal{J}_{\sigma}(X,\Delta;\|D\|)=\mathcal{J}(X,\Delta;\|(1+\varepsilon)D+\varepsilon' A\|),
$$
where $A$ is an ample divisor on $X$.
\end{rmk}

Let us prove the three main theorems on the diminished multiplier ideal.

\begin{proof}[Proof of \thref{imain1}]
Let $\mathfrak{q}$ be an ideal in $\O_X$. and let $\nu\in \mathrm{Val}^*_X$ be a valuation.

\medskip

\noindent\emph{(1) From $\mathcal{J}_{\sigma}\subseteq \mathcal{J}$} 

By \thref{mag2}, there exists $0<\varepsilon\ll 1$ such that if $\mathfrak{q}\subseteq \mathcal{J}_{\sigma}(X,\Delta;\|D\|)$, then
$$ \nu(\mathfrak{q})\ge \sigma_{\nu}((1+\varepsilon)D)-A_{X,\Delta}(\nu).$$
Using (b) of \thref{mag?}, we have $\sigma_{\nu}((1+\varepsilon))=(1+\varepsilon)\nu(\|D\|)$, so 
$$ \nu(\mathfrak{q})\ge(1+\varepsilon)\nu(\|D\|)-A_{X,\Delta}(\nu)>\nu(\|D\|)-A_{X,\Delta}(\nu).$$
Furthermore, $\nu(\|D\|)\ne 0$ in this context (otherwise the inequality would be trivial). Then by \thref{mag}, it follows that $\mathfrak{q}\subseteq \mathcal{J}(X,\Delta;\|D\|)$. Hence we have one containment:
$$ \mathcal{J}_{\sigma}(X,\Delta;\|D\|)\subseteq \mathcal{J}(X,\Delta;\|D\|).$$

\medskip

\noindent \emph{(2) From $\mathcal{J}\subseteq \mathcal{J}_{\sigma}$} 

Let $A$ be an ample divisor on $X$. By \thref{KR}, for every $0<\varepsilon\ll 1$, we know 
$$\mathcal{J}(X,\Delta;\|D\|)=\mathcal{J}(X,\Delta;\|(1+\varepsilon)D\|).$$
Thus, applying \thref{mag} again shows that if $\mathfrak{q}\subseteq \mathcal{J}(X,\Delta;\|D\|)$, we must have
$$ \nu(\mathfrak{q})>\nu(\|(1+\varepsilon)D\|)-A_{X,\Delta}(\nu)\,\,\text{ for all }\nu\in \mathrm{Val}^*_X.$$
Now, pick an even smaller $0<\varepsilon'\ll \varepsilon\ll 1$. Then
$$ \nu(\mathfrak{q})>\nu(\|(1+\varepsilon)D+\varepsilon' A\|)-A_{X,\Delta}(\nu)\,\,\text{ for all }\nu\in \mathrm{Val}^*_X,$$
so by \thref{mag} once more,
$$ \mathfrak{q}\subseteq \mathcal{J}(X,\Delta;\|(1+\varepsilon)D+\varepsilon' A\|)$$
Finally, by \thref{rmk}, $\mathfrak{q}\subseteq \mathcal{J}_{\sigma}(X,\Delta;\|D\|)$. Thus the reverse containment is established.
\end{proof}

\begin{proof}[Proof of \thref{imain2}]
By \cite[Theorem 1.2 and Corollary 5.2]{CD13}, for a big divisor $D$, we have
$$ \B_-(D)=\bigcup_mZ(\mathcal{J}(X,\Delta;\|mD\|))_{\mathrm{red}}.$$

Thus, by \thref{diminished},
$$
\begin{aligned}
\B_-(D)&=\bigcup_m \B_-\left(D+\frac{1}{m}A\right)
\\ &=\bigcup_m \bigcup_{m'}Z\left(\mathcal{J}\left(X,\Delta;\left\|m'D+\frac{m'}{m}A\right\|\right)\right)_{\mathrm{red}}
\\ &=\bigcup_{m'}\bigcup_mZ\left(\mathcal{J}\left(X,\Delta;\left\|m'D+\frac{m'}{m}A\right\|\right)\right)_{\mathrm{red}}
\\ &=\bigcup_{m'}Z(\mathcal{J}_-(X,\Delta;\|D\|))_{\mathrm{red}}.
\end{aligned}
$$
Moreover, the containments
$$ \mathcal{J}_{\sigma}(X,\Delta;\|(m+1)D\|)\subseteq \mathcal{J}_-(X,\Delta;\|(m+1)D\|)\subseteq \mathcal{J}_{\sigma}(X,\Delta;\|mD\|)$$
yield the conclusion.
\end{proof}

\begin{proof}[Proof of \thref{imain3}]
First, suppose that $D$ is big. In this case, $\mathcal{J}_\sigma(X,\Delta;\|D\|)$ agrees with $\mathcal{J}(X,\Delta;\|D\|)$ by \thref{imain1}, and the vanishing statement follows directly from \thref{Nadel}. Hence, the result holds when $D$ is big.

We now proceed by induction on $\dim X-\kappa_{\sigma}(D)$. Let $H$ be a very ample divisor on $X$, and let $H'\in |H|$ be a very general element of $|H|$. By \cite[Theorem 9.5.16]{Laz04}, we have:
$$ \mathcal{J}_{\sigma}(X,\Delta;\|D\|)|_{H'}=\mathcal{J}_{\sigma}(H',\Delta|_{H'};\|D|_{H'}\|).$$
Since $(X,\Delta)$ is klt, the pair $(H',\Delta|_{H'})$ is klt by the adjunction formula (cf. \cite[Lemma 5.17]{KM98}). 

Consider the following exact sequence
$$ 
\begin{aligned}
    0\to \O_X(L)\otimes \mathcal{J}_{\sigma}(X,\Delta;\|D\|)&\to \O_X(L+H)\otimes \mathcal{J}_{\sigma}(X,\Delta;\|D\|)
    \\ &\to \O_{H'}((L+H)|_{H'})\otimes \mathcal{J}_{\sigma}(H',\Delta|_{H'};\|D|_{H'}\|)\to 0.
\end{aligned}$$
The exactness of the left holds because the support of $\mathcal{T}or^{\O_X}_1(\mathcal{J}_{\sigma}(H',\Delta|_{H'};\|D|_{H'}\|),\O_{H'})$ is contained in $H'$. Consequently, the natural map
$$\mathcal{T}or^{\O_X}_1(\mathcal{J}_{\sigma}(H',\Delta|_{H'};\|D|_{H'}\|),\O_{H'}((L+H)|_{H'})\to \mathcal{J}_{\sigma}(X,\Delta;\|D\|)\otimes \O_X(L)$$
is zero. 

By \cite[2.7 Proposition (5) in Chapter V]{Nak04}, we have $\kappa_{\sigma}(D)\le \kappa_{\sigma}(D|_{H'})$. Moreover, by the adjunction formula (\cite[(4.2.9), (4.2.10), and Proposition 4.5 (1)]{Kol13}), we see $(L+H)|_{H'}\sim_{\Q}K_{H'}+\Delta|_{H'}+D|_{H'}$ if we use the adjunction theorem. Thus, when we take cohomology of the above short exact sequence, we can apply the induction hypothesis to the restricted divisor $D|_{H'}$, and then use \thref{Nadel} to the middle cohomology. This completes the inductive step and hence the proof.
\end{proof}

Let us prove an important corollary that is used to prove \thref{ithm}.

\begin{coro}\thlabel{corocoro}
Let $(X,\Delta)$ be a $\Q$-factorial projective klt pair such that $\kappa_{\sigma}(K_X+\Delta)\ge \dim X-1$. Suppose every irreducible component of $\B_-(K_X+\Delta)$ has dimension at most $1$. Then for any positive integer $m$ such that $m(K_X+\Delta)$ is Cartier,
$$ H^i(X,\O_X(m(K_X+\Delta)))=0\text{ for all }i\ge 2.$$
\end{coro}

\begin{proof}
Since $\kappa_{\sigma}(K_X+\Delta)\ge \dim X-1$ and $m(K_X+\Delta)=K_X+\Delta+(m-1)(K_X+\Delta)$, by \thref{imain3}, we get
$$ H^i(X,\O_X(m(K_X+\Delta))\otimes \mathcal{J}_{\sigma}(X,\Delta;\|(m-1)(K_X+\Delta)\|))=0\text{ for }i\ge 2.$$

Next, by \thref{imain2}, 
$$Z(\mathcal{J}_{\sigma}(X,\Delta;\|(m-1)(K_X+\Delta)\|))_{\mathrm{red}}\subseteq \B_-(K_X+\Delta).$$ By hypothesis, $\dim Z(\mathcal{J}_{\sigma}(X,\Delta;\|(m-1)(K_X+\Delta)\|))_{\mathrm{red}}\le 1$. Hence
$$ H^i(X,\O_X(m(K_X+\Delta))\otimes \O_X/\mathcal{J}_{\sigma}(X,\Delta;\|(m-1)(K_X+\Delta)\|))=0\text{ for }i\ge 2.$$

Finally, consider the short exact sequence
$$ 0\to \mathcal{J}_{\sigma}(X,\Delta;\|(m-1)(K_X+\Delta)\|)\to \O_X\to \O_X/\mathcal{J}_{\sigma}(X,\Delta;\|(m-1)(K_X+\Delta)\|)\to 0,$$
tensor it with $\O_X(m(K_X+\Delta))$, and take cohomology. Since both of the higher cohomology groups we identified above vanish for $i\ge 2$, it follows that
$$ H^i(X,\O_X(m(K_X+\Delta)))=0\text{ for all }i\ge 2.$$
This concludes the proof.
\end{proof}

\section{Proof of the main theorem on termination of flips}

The goal of this section is to prove \thref{ithm}. 

\begin{proof}[Proof of \thref{ithm}]
Let us divide the proof into several steps.
\medskip

\noindent\textbf{Step 1}. For each $i$, let
$$
\begin{tikzcd}
X_{i-1}\ar[dashed,"f_i"]{rr}\ar["\varphi_{i1}"']{rd}& &X_i \ar["\varphi_{i2}"]{ld}\\
& Z_i&
\end{tikzcd}
$$
be the flipping diagram of $f_i$. By \thref{KRss}, there exist common resolutions $g_{i1}:X'_i\to X_{i-1}$, and $g_{i2}:X'_i\to X_i$ such that the following diagram commutes:
$$
\begin{tikzcd}
& X'_1\ar["g_{11}"']{ld}\ar["g_{12}"]{rd}& & X'_2\ar["g_{21}"']{ld}\ar["g_{22}"]{rd}& & & & X'_n \ar["g_{n1}"']{ld}\ar["g_{n2}"]{rd}&\\
X_0\ar[dashed,"f_1"]{rr}\ar["\varphi_{11}"']{rd}& &X_1\ar["\varphi_{12}"]{ld}\ar[dashed,"f_2"]{rr}\ar["\varphi_{21}"']{rd}& &X_2\ar["\varphi_{22}"]{ld}& \cdots & X_{n-1} \ar["\varphi_{n1}"']{rd}\ar[dashed,"f_n"]{rr}& & X_n.\ar["\varphi_{n2}"]{ld}\\
& Z_1& & Z_2& & & & Z_n &
\end{tikzcd}
$$
Furthermore, by \thref{forapp}, we have $\kappa_{\sigma}(K_{X_i}+\Delta_i)\ge \dim X-1$ for every $i$, and by \thref{diminished2}, each irreducible component of $\B_-(K_{X_i}+\Delta_i)$ has dimension at most $1$.

\medskip

\noindent\textbf{Step 2}. By (b), there is a positive integer $m$ such that $m(K_{X_i}+\Delta_i)$ is Cartier for each $i$. Set
$$ D_i:=2m(K_{X_i}+\Delta_i)\text{ and }d_i:=h^i(X_i,\O_{X_i}(D_i)).$$
Since $X$ has rational singularities (cf. \cite[Theorem 5.22]{KM98}), the Leray spectral sequence shows that
$$ H^s(X_{i-1},\O_{X_{i-1}}(D_{i-1}))=H^s(X'_i,\O_{X'_i}(g^*_{i1}D_{i-1}))$$
By the Negativity lemma, there is an effective $g_{i2}$-exceptional divisor $E_i$ on $X'_i$ such that
\begin{equation}\label{2} 
g^*_{i1}D_{i-1}=g^*_{i2}D_i+E_i
\end{equation}
Next, consider the Leray spectral sequence
$$ E^{st}_2=H^s(X_i,R^tg_{i2*}\O_{X'_i}(g^*_{i1}D_{i-1}))\implies H^{s+t}(X'_i,\O_{X'_i}(g^*_{i1}D_{i-1})).$$
An inspection of its second page shows there is a natural injection
$$ 
\begin{aligned}
H^1(X_i,\O_{X_i}(D_i))&=H^1(X_i,g_{i2*}\O_{X'_i}(D_{i-1}))
\\ &\to H^1(X'_i,\O_{X'_i}(g^*_{i1}D_{i-1}))=H^1(X_{i-1},\O_{X_{i-1}}(D_{i-1})),
\end{aligned}$$
Hence we conclude that $d_i\le d_{i-1}$ for all $i$.

\medskip

\noindent\textbf{Step 3}. Assume that the MMP does not terminate. Since $\{d_i\}$ is a non-increasing sequence of positive integers, for sufficiently large $i_0$ we must have $d_{i_0-1}=d_{i_0}$.

Next, consider the following Leray spectral sequences:
$$
\begin{aligned}
{}^{'}E^{st}_2&=H^s(Z_{i_0},R^t\varphi_{i_01*}\O_{X_{i_0-1}}(D_{i_0-1}))\implies H^{s+t}(X_{i_0-1},\O_{X_{i_0-1}}(D_{i_0-1})),\text{ and } \\
{}^{''}E^{st}_2&=H^s(Z_{i_0},R^t(\varphi_{i_01}\circ g_{i_01})_*\O_{X'_{i_0-1}}(D_{i_0-1}))\implies H^{s+t}(X'_{i_0},\O_{X'_{i_0}}(g^*_{i_01}D_{i_0-1})).
\end{aligned}
$$
Because $X$ has rational singularities (cf. \cite[Theorem 5.22]{KM98}), the natural map 
$$\O_{X_{i_0-1}}(D_{i-1})\cong \O_{X'_{i_0}}(g^*_{i_01}D_{i_0-1})$$
induces isomorphisms ${}^{'}E^{st}_2\cong {}^{''}E^{st}_2$ for all $s,t$. Thus, these two spectral sequences are isomorphic. From the resulting edge map
$$ \alpha_{1}:H^1(Z_{i_0},\varphi_{i_01*}\O_{X_{i_0-1}}(D_{i_0-1}))\to H^1(X_{i_0-1},\O_{X_{i_0-1}}(D_{i_0-1}))$$
we can construct an injection
$$ \alpha_{2}:H^1(Z_{i_0},(\varphi_{i_02}\circ g_{i_02})_*\O_{X'_{i_0}}(g^*_{i_01}D_{i_0-1}))\to H^1(X_{i_0-1},\O_{X_{i_0-1}}(D_{i_0-1})),$$
using the commutative diagram from earlier steps. By \eqref{2}, this in turn yields
$$ \alpha_{3}:H^1(Z_{i_0},\varphi_{i_02*}\O_{X_{i_0}}(D_{i_0}))\to H^1(X_{i_0-1},\O_{X_{i_0-1}}(D_{i_0-1})).$$
Observe that 
$$D_{i_0}=(K_{X_{i_0}}+\Delta_{i_0})+(D_{i_0}-(K_{X_{i_0}}+\Delta_{i_0}))$$
and 
$$D_{i_0}-(K_{X_{i_0}}+\Delta_{i_0})=(2m-1)(K_{X_{i_0}}+\Delta_{i_0})$$
which is $\varphi_{i_02}$-ample. By the relative Kawamata-Viehweg vanishing theorem (cf. \cite[Theorem 1-2-3]{KMM87}) and the Leray spectral sequence, we obtain an injection
$$ \alpha_{4}:H^1(X_{i_0},\O_{X_{i_0}}(D_{i_0}))\to H^1(X_{i_0-1},\O_{X_{i_0-1}}(D_{i_0-1})).$$
Moreover, $\alpha_1$ is an isomorphism if and only if $\alpha_4$ is an isomorphism. Since $d_{i_0-1}=d_{i_0}$, we conclude that $\alpha_4$, and hence $\alpha_1$ is indeed an isomorphism.

\medskip

\noindent\textbf{Step 4}. The first four terms of the five-term exact sequence of ${}^{'}E^{st}_2$ are
$$ 
\begin{aligned}
0&\to H^1(Z_{i_0},\varphi_{i_01*}\O_{X_{i_0-1}}(D_{i_0-1}))\overset{\alpha_{1}}{\to}H^1(X_{i_0-1},\O_{X_{i_0-1}}(D_{i_0-1})) \\ &\to H^0(Z_{i_0},R^1\varphi_{i_01*}\O_{X_{i_0-1}}(D_{i_0-1}))\to H^2(Z_{i_0},\varphi_{i_01*}\O_{X_{i_0-1}}(D_{i_0-1})).
\end{aligned}
$$
Arguing as in Step 3, we compute:
$$
\begin{aligned}
H^2(Z_{i_0},\varphi_{i_01*}\O_{X_{i_0-1}}(D_{i_0-1}))&=H^2(Z_{i_0},(\varphi_{i_01}\circ g_{i_01})_*\O_{X'_{i_0}}(g^*_{i_01}D_{i_0-1}))
\\ &\overset{(1)}{=}H^2(Z_{i_0},(\varphi_{i_02}\circ g_{i_02})_*\O_{X'_{i_0}}(g^*_{i2}D_{i_0}+E_{i_0}))
\\ &=H^2(Z_{i_0},\varphi_{i_02*}\O_{X_{i_0}}(D_{i_0}))
\\ &\overset{(2)}{=}H^2(X_{i_0},\O_{X_{i_0}}(D_{i_0}))\overset{(3)}{=}0,
\end{aligned}$$
where $(1)$ follows from the commutative diagram and \eqref{2}, $(2)$ follows from the relative Kawamata--Viehweg vanishing theorem combined with the Leray spectral sequence, and $(3)$ holds by the argument in the last sentence of Step 1, $(c)$, together with \thref{corocoro}.

Since $\alpha_1$ is an isomorphism, it follows that
$$ H^0(Z_{i_0},R^1\varphi_{i_01*}\O_{X_{i_0-1}}(D_{i_0-1}))=0.$$
By (a), the map $f_{i_0}$ is a flip of type $(1,\dim X-2)$. Hence the support of
$$R^1\varphi_{i_01*}\O_{X_{i_0-1}}(D_{i_0-1})$$
is zero-dimensional, which forces 
$$R^1\varphi_{i_01*}\O_{X_{i_0-1}}(D_{i_0-1})=0.$$

\medskip

\noindent\textbf{Step 5}. In this step, we show that there exists a rational curve $C\subseteq X_{i_0-1}$ such that
$$H^1(C,\O_{X_{i_0-1}}(D_{i_0-1})|_C)=0.$$
By the Cone Theorem (cf. \cite[Theorem 1]{Kaw91}), there is a rational curve $C\subseteq X_{i_0-1}$ that is contracted by $\varphi_{i_01}$. Let $x:=\varphi_{i_01}(C)$, and let $E:=\varphi^{-1}_{i_01}(x)$ be the scheme-theoretic fiber of $\varphi_{i_01}$ over $x$. Since $C\subseteq E\subseteq X_{i_0-1}$, we can construct a coherent sheaf $\mathcal{F}$ on $X_{i_0-1}$ fitting into an exact sequence
$$ 0\to \mathcal{F}\to \O_{X_{i_0-1}}(D_{i_0-1})\to \O_{X_{i_0-1}}(D_{i_0-1})|_C\to 0.$$
Because every fiber of $\varphi_{i_01}$ is $1$-dimensional, we have $R^2\varphi_{i_01*}\mathcal{F}=0$. Consequently, thee is a surjection
$$ 0=R^1\varphi_{i_01*}\O_{X_{i_0-1}}(D_{i_0-1})\twoheadrightarrow  R^1\varphi_{i_01*}\O_{X_{i_0-1}}(D_{i_0-1})|_C,$$
whoch implies
$$ H^1(C,\O_{X_{i_0-1}}(D_{i_0-1})|_C)=R^1\varphi_{i_01*}\O_{X_{i_0-1}}(D_{i_0-1})|_C=0.$$

\medskip

\noindent\textbf{Step 6}. Since $-(K_{X_{{i_0}-1}}+\Delta_{i_0-1})$ is $\varphi_{i_01}$-ample, it follows that 
$$\mathcal{L}:=\O_{X_{i_0-1}}(m(K_{X_{i_0-1}}+\Delta_{i_0-1}))|_C$$
is an anti-ample line bundle on $C$. Let $\nu:\P^1\to C$ be the normalization. As $\nu$ is finite and pullback of an ample line bundle along a finite morphism remains ample (cf. \cite[Lemma 0B5V]{Stacks}), we see that $\nu^*\mathcal{L}$ is anti-ample on $\P^1$. Hence $\O_{\P^1}(-\ell)=\nu^*\mathcal{L}$ for some positive integer $\ell$. Recalling that
$$ D_{i_0-1}=2m(K_{X_{i_0-1}}+\Delta_{i_0-1}),$$
we obtain $\O_{X_{i_0-1}}(D_{i_0-1})|_C=\O_{\P^1}(-2\ell)$. Moreover, since any finite morphism has vanishing higher direct images (cf. \cite[Lemma 02OE]{Stacks}), the Leray spectral sequence gives
$$ H^1(\P^1,\O_{\P^1}(-2\ell))=H^1(C,\O_{X_{i_0-1}}(D_{i_0-1})|_C)=0,$$
which contradicts the well-known fact that
$$H^1(\P^1,\O_{\P^1}(-N))\ne 0\text{ for all }N\ge 2.$$
This contradiction forces the MMP to terminate, thereby completing the proof of the assertion.
\end{proof}

\begin{rmk}
It is worth noting that Step 5 and 6 in the proof of \thref{ithm} are similar to the proof of \cite[Theorem 1]{Bro99}.
\end{rmk}

\end{document}